\documentclass[12pt]{amsart}
\usepackage{amssymb, amscd, amsmath, amsthm, amscd, curves,
epsf, epsfig, latexsym, enumerate, pstricks,graphics,tikz-cd,url}

\newtheorem{theorem}{Theorem}
\newtheorem{lemma}[theorem]{Lemma}
\newtheorem{corollary}[theorem]{Corollary}
\newtheorem{proposition}[theorem]{Proposition}
\theoremstyle{definition}
\newtheorem{definition}[theorem]{Definition}
\newtheorem{remark}[theorem]{Remark}
\newtheorem{example}[theorem]{Example}
\newtheorem*{question}{Question}
\newtheorem{maintheorem}{Theorem}
\newtheorem{maincorollary}[maintheorem]{Corollary}

\newtheorem{conjecture}[maintheorem]{Conjecture}

\newcommand{\g}{\Gamma}

\newcommand{\R}{\mathbb R}
\newcommand{\Z}{\mathbb Z}

\DeclareMathOperator{\ev}{ev}
\DeclareMathOperator{\Ext}{Ext}
\DeclareMathOperator{\Hom}{Hom}

\DeclareMathOperator{\id}{id}

\DeclareMathOperator{\rel}{rel}

\DeclareMathOperator{\Wh}{Wh}

\newcommand{\ol}[1]{{\overline{#1}}}

\newcommand{\wt}[1]{{\widetilde{#1}}}
\newcommand{\cC}{\mathcal C}
\newcommand{\cD}{\mathcal D}
\newcommand{\cM}{\mathcal M}

\newcommand{\cS}{\mathcal S}

\newcommand{\bfK}{\mathbf K}
\newcommand{\bfL}{\mathbf L}

\newcommand{\px}{\pi_1X} 
\newcommand{\zpx}{\Z\pi_1X} 
\newcommand{\zp}{\Z\pi}

\begin{document}. 

\title{Aspherical manifolds with boundary}

\author{James F. Davis \and J. A. Hillman }
\address{{Department of Mathematics, Indiana University,}
\newline
{Bloomington, IN 47405 USA} 
\newline
{School of Mathematics and Statistics, University of Sydney,}
\newline
Sydney,  NSW 2006, Australia }

\email{jfdavis@iu.edu, jonathan.hillman@sydney.edu.au}

\begin{abstract}
We undertake a systematic investigation of compact aspherical manifolds with boundary; motivated by the plethora of examples in the bounded case and by the beauty of the theory in the closed case.  
Our main theorems give a homological criterion for when a closed manifold, together with maps from the fundamental groups of its components to a fixed group, can be realized as the boundary of a compact aspherical manifold.   This is done in two steps: we first produce a Poincar\'e pair and then apply surgery theory to obtain a manifold.   We illustrate this in the case of abelian fundamental group.   The results of this paper will be applied in a sequel where we classify compact aspherical 4-manifolds with elementary amenable fundamental group.
\end{abstract}



\maketitle

This topic of this paper is the homeomorphism classification of  compact aspherical topological  manifolds with boundary.   
An {\em aspherical space} is a space whose universal cover is contractible.  Here are two obvious comments: the boundary of an aspherical manifold need not be aspherical and the fundamental group of an aspherical manifold is torsion-free.  

The {\em  Borel Conjecture for a compact aspherical manifold $M$} states that any homotopy equivalence $H : W \to M$ from a compact manifold $W$ which restricts to a homeomorphism $h : \partial W \xrightarrow{\cong} \partial M$ on the boundary is homotopic to a homeomorphism rel $\partial W$, i.e. the homotopy is fixed on  $\partial W$.

\begin{maintheorem} \label{BC}
Let $M$ be a compact aspherical $n$-manifold with boundary, with fundamental group $\pi$, and with dimension $n$.   
\begin{enumerate}
 \item If $n = 4$ and $\pi$ is elementary amenable, then the Borel Conjecture holds for $M$.
 \item If $n \geq 5$ and the Farrell-Jones Conjecture (FJC) holds in $K$- and $L$-theory for $\pi$, then the Borel Conjecture holds for $M$.
\end{enumerate}
\end{maintheorem}

Case (2) is well-known to the experts (although perhaps not well-documented when the manifold has a boundary), but is new in case (1).   Case (1) is due to  the fact that topological surgery in dimension 4 works for elementary amenable fundamental groups, and the validity of the Farrell-Jones Conjecture for the groups in Theorem \ref{BC}, which in turn depends on our classification of fundamental groups of compact aspherical 4-manifolds with elementary amenable fundamental group from \cite{DH2}.  
The focus on elementary amenable groups is motivated by the fact that the class of fundamental groups of compact aspherical 4-manifolds for which the Disc Embedding Conjecture \cite{BKKPR} is known is contained in the class of  elementary amenable groups  (see the introduction to \cite{DH2}).

In practice,  one is often given compact aspherical manifolds $W$ and $M$ with isomorphic  fundamental groups  and homeomorphic boundaries and asks if the manifolds themselves are homeomorphic.   This question is studied using the Borel Conjecture and elementary obstruction theory and is discussed in Subsection \ref{subsection_hb}.   An example where the answer is no is given in Example \ref{nonhomeomorphic}.

A space $X$ is {\em finitely dominated} if it has the homotopy type of a CW-complex and if it is the homotopy retract of a finite CW-complex, i.e., there is a finite complex $F$ and there are maps $i: X \to F$ and $r : F \to X$ so that $r \circ i \simeq \id_X$.  In particular a finite CW-complex is finitely dominated.

We formulate the following conjecture which is the main focus of this paper.   It is an existence variant of the Borel Conjecture.

%

\begin{conjecture} \label{EC}
Let $(X,Y)$ be a pair of spaces where $X$ be a finitely dominated aspherical space with fundamental group $\pi$ and $Y$   is a closed nonempty $(n-1)$-dimensional manifold. If $H^i(X,Y;\Z\pi)$ vanishes for $i \not = n$ and is infinite cyclic (as an abelian group) for $i = n$, then there is a compact manifold $M$ and a commutative square
$$
\begin{CD}
\partial M @>\cong>> Y\\
@VVV @VVV \\
M @>\sim>> X
\end{CD}
$$
so that the upper horizontal map is a homeomorphism and the lower horizontal map is a homotopy equivalence.   In particular, $Y$ is the boundary of a compact aspherical manifold with fundamental group $\pi$.

\end{conjecture}

Recall that if $K(\pi,1)$ is a finite complex and if $H^i(\pi; \Z\pi)$ vanishes for $i \not = n$ and is infinite cyclic when $i = n$, then $K(\pi,1)$ is a Poincar\'e complex of dimension $n$ (see \cite{JW}, \cite{Bieri}, or \cite{Brown}).  

The following theorem is an analogue of this and is the first step towards the resolution of Conjecture \ref{EC}.

\begin{maintheorem} \label{PP}
Suppose $\pi$ is a group and $K(\pi,1)$ is finitely dominated.   Let $Y$ be an $(n-1)$-dimensional Poincar\'e complex with a map $c : Y \to K(\pi,1)$.     If $H^i(K(\pi,1),Y; \Z\pi)$ vanishes for $i \not = n$ and is infinite cyclic when $i = n$, then $(K(\pi,1), Y)$ is a Poincar\'e pair of dimension $n$.  
 \end{maintheorem}
 
 In contrast to \cite{BE}, we assume neither that $Y$ is aspherical nor that $\pi_1(c)$ is injective.

A Poincar\'e pair $(X,Y)$ with $Y$  a closed manifold can be {\em  realized rel $Y$} if there is a commutative square 
$$
\begin{CD}
\partial M @>\cong>> Y\\
@VVV @VVV \\
M @>\sim>> X
\end{CD}
$$
where $M$ is a compact manifold, the upper horizontal map is a homeomorphism and the lower horizontal map is a homotopy equivalence. 

\begin{maintheorem} \label{ET}
Let $(X,Y)$ be a Poincar\'e pair of dimension $n$ with $X$ aspherical  with fundamental group $\pi$ and with $Y$ a nonempty closed manifold.
\begin{enumerate}
\item If $n \leq 2$, then $(X,Y)$ can be realized rel $Y$.
 \item If $n = 4$ and $\pi$ is elementary amenable, then $(X,Y)$ can be realized rel $Y$.
 \item If $n \geq 5$ and the Farrell-Jones Conjecture (FJC) holds in $K$- and $L$-theory for $\pi$, then $(X,Y)$ can be realized rel $Y$.
\end{enumerate}
\end{maintheorem}

The hypothesis that $Y$ is a nonempty closed manifold  in Conjecture \ref{EC} is crucial for the proof of Theorem \ref{ET}(3).   In Lemma \ref{E_false_PP}, we give an example of a Poincar\'e pair $(X,Y)$ with $X$ aspherical which does not have the homotopy type of a compact manifold pair.

The analogue of Conjecture \ref{EC} for empty $Y$ is a famous problem of Wall \cite[page 391]{Wall79} ``Is every Poincar\'e duality group $\Gamma$ the fundamental group of a closed $K(\Gamma,1)$ manifold?"  Michael Davis \cite[Theorem C]{Davis98} gave a negative answer by constructing Poincar\'e duality groups which are not finitely presented.   There are no negative answers known to the modified question: Is every finitely presented Poincar\'e duality group $\Gamma$ the fundamental group of a closed $K(\Gamma,1)$ manifold?   Geometric group theory \cite{Davis00} has produced fascinating examples, but there is no systematic way of attacking this question, which is why it is probably best left as a question and not a conjecture.

Theorems \ref{PP} and \ref{ET} show that the existence question for $K(\pi,1)$ manifolds with boundary is essentially homological.  We work out some special cases, for example, classifying homotopy tori ($\pi= \Z^k$) and homotopy figure eights ($\pi= F(2)$).

\begin{maintheorem} \label{duality}
 Let $(X,Y)$ be a  CW-pair with $X$ aspherical and finitely dominated, $\pi = \pi_1X$ a $k$-duality group (i.e.~$H^i(X;\Z\pi) = 0$ for $i \not = k$), and $Y$ a nonempty Poincar\'e complex of dimension $n-1$.  Let $D = H^k(\pi;\Z\pi)$ (this right $\Z\pi$-module is called the dualizing module).  Let $\ol X$ and $\ol Y$ be the induced $\pi$-covers of $X$ and $Y$.   Then $(X,Y)$ is a Poincar\'e pair if and only if there exists a homomorphism $w : \pi \to \{\pm 1\}$ which extends $w_1Y$ in the sense that $w_1Y = w \circ \pi_1(Y \hookrightarrow X)$ and  \begin{enumerate}
 \item $H_i\ol Y $ vanishes for $i \not = 0,n-1-k$; and 
 \item  $\ker (H_{n-1-k} \ol Y  \to H_{n-1-k} \ol X) \cong \ol D$ as left $\Z\pi$-modules, where we twist by $w: \pi \to \{\pm 1\}$.  
\end{enumerate}
\end{maintheorem}

After specializing to  the case where $\pi$ is finitely generated free abelian, we obtain the following corollary.

\begin{maincorollary} \label{abelian}
\begin{enumerate}
\item The Borel Conjecture holds for compact aspherical manifolds (possibly with boundary) with fundamental group $\Z^k$.
 \item Conjecture \ref{EC} holds when $\pi = \Z^k$.
 \item If $M$ is a compact aspherical $n$-manifold with fundamental group $\Z^k$, then $H_*(\partial \wt M) \cong H_*(S^{n-k-1})$.
 \item If $N$ is a closed $(n-1)$-manifold, $\ol N \to N$ is a regular $\Z^k$-cover, and $H_*(\ol N) \cong H_*(S^{n-k-1})$, then $N$ is the boundary of a compact aspherical $n$-manifold with fundamental group $\Z^k$ and induced $\pi$-cover $\ol N \to N$.
\end{enumerate}
\end{maincorollary}

Section \ref{definitions} reviews the definitions of Poincar\'e complexes,  Poincar\'e pairs, and Poincar\'e duality groups.     Section \ref{pdgb_section} is pure algebraic topology; it introduces the theory of Poincar\'e duality groups with boundary and includes the proof of Theorem \ref{PP}.    Section \ref{uniqueness} deals with the  Borel Conjecture for compact aspherical manifolds with boundary and gives the proof of Theorem \ref{BC}.   Section \ref{existence}  deals with the realization question for Poincar\'e pairs $(X,Y)$ where $X$ is aspherical and $Y$ is a closed manifold and gives the proof of Theorem \ref{ET}.  Section \ref{examples} gives examples and applications and gives the proof of Theorem \ref{abelian}.

\medskip
\noindent{\bf Acknowledgment.}
This collaboration began at the conference on
 {\it Topology of Manifolds : Interactions between high and low dimensions\/} 
 held at Creswick, VIC.
The authors would like to thank the 
 MATRIX Institute and the  the National Science Foundation for their support.  
 JFD would like to thank the Simons Foundation for its support under  the Simons Collaboration Grant 713226.  The authors would like to thank Shmuel Weinberger for useful conversations.

\section{Poincar\'e complexes, Poincar\'e pairs, and Poincar\'e duality groups} \label{definitions}

This section  reviews standard definitions and sets notation.   Path-connected spaces will assumed to be based, although this is only for expositional convenience, so that we can work with the fundamental group rather than the fundamental groupoid.   The action of the fundamental group on the universal cover is assumed to be a left action.   
  If $B$ is a  $\Z\pi_1X$-module, define {\em (co)homology with twisted coefficients} to be the abelian groups 
\begin{align*}
H_i(X;B) &= H_i(B \otimes_{\zpx} S_\bullet \wt X)  \\
H^i(X;B) &= H^i(\Hom_{\zpx}(S_\bullet \wt X, B))
\end{align*}
where in homology, $B$ is a right module and in cohomology, $B$ is a left module, and $S_\bullet \wt X$ is the singular chain complex of the universal cover of $X$.

An {\em orientation character} is a homomorphism $w: \px \to \Z^{\times} = \{\pm 1 \}$.  Define an involution (an antiautomorphism of order 2) on $\Z\pi_1X$ by $\ol{\sum a_g g} = \sum a_g w(g) g^{-1}$. If $B$ is a left (or right) $\zpx$-module, define $\ol B$ to be the right (or left) $\zpx$-module whose underlying abelian group  is $B$ with $\zpx$-action given by $br:= \ol r b  $ (or $rb:= b \ol r$).  If $B$ is a $(\Z\pi_1X,\Z\pi_1X)$-bimodule, then so is $\ol B$ with $rbr' :=  \ol {r'} b \ol r$.

The notions of (co)homology with twisted coefficients and orientation character can be extended to spaces which are not path-connected by choosing a base point in each path component.    If $X$ has path components $\{X_\alpha\}$, then a $\zpx$-module $B = \{B_\alpha\}$ means a $\zpx_\alpha$-module $B_\alpha$ for each $\alpha$ and an orientation character $w = \{w_\alpha\}$ means a orientation character for each component.   Then define $H_i(X;B) := \oplus H_i(X_\alpha; B_\alpha)$ and $H^i(X;B) := \prod H^i(X_\alpha; B_\alpha)$.

Poincar\'e complexes and pairs are due to Wall \cite{Wall_P}, \cite{Wall_S}; a recent account is given in \cite{KQS}.   We follow the conventions of \cite{Hi20}.

\begin{definition}
 An {\em $n$-dimensional Poincar\'e complex} (a $PD_n$-complex) $(X,w)$ is a space $X$ with an orientation character $w$ so that
 
\begin{itemize}
 \item $X$ is finitely dominated and
 \item there exists  $[X] \in H_n(X; \ol \Z)$ ``the fundamental class" so that for all $i$,
 $$
 \cap [X] : \ol{H^i(X; \zpx)} \xrightarrow{\cong} H_{n-i}(X; \ol{\zpx}) 
 $$
 is an isomorphism of left $\Z\pi_1X$-modules.
\end{itemize}
\end{definition}

These two conditions can be replaced by 
\begin{itemize}
 \item $X$ has the homotopy type of a CW-complex, 
 \item $\pi_1X$ is finitely presented, and 
 \item there exists  $[X] \in H_n(X; \ol \Z)$ so that for any $i$ and for any left $\Z\pi_1X$-module $B$
 $$
 \cap [X] : \ol{H^i(X; B)} \xrightarrow{\cong} H_{n-i}(X; \ol{B})
 $$
 is an isomorphism of abelian groups.
\end{itemize}
To see why the top two bullet points imply the bottom three, one needs Lemma 1.1 of \cite{Wall_P}.  To see why the bottom three bullet points imply the top two, one needs Corollary 2 of \cite{Browder} or the Corollary on page 135 of \cite{Brown75}.

For a $(\Z\pi_1X,\Z\pi_1X)$-bimodule $B$, 
 $$
 \cap [X] : \ol{H^i(X; B)} \xrightarrow{\cong} H_{n-i}(X; \ol{B}) 
 $$
is an isomorphism of  left $\Z\pi_1X$-modules.

The map $\Z\pi_1X \to \ol{\Z\pi_1X}, \quad \sum a_g g \mapsto \sum a_g w(g) g^{-1}$ is an isomorphism of  $(\Z\pi_1X,\Z\pi_1X)$-bimodules and thus
$$
\ol{H^i(X; \zpx)} \xrightarrow{\cong} H_{n-i}(X; \zpx), 
$$
as left $\Z\pi_1X$-modules.   This is often  convenient to use.

If $w$ is trivial on the kernel of a group homomorphism $\pi_1X \to G$, one also sees that 
$$
\ol{H^i(X; \Z G)} \xrightarrow{\cong} H_{n-i}(X; \Z G). 
$$
as left $\Z G$-modules.

Note that a closed topological manifold is a Poincar\'e complex; this is one reason we did not assume that $X$ itself is a CW-complex.    For $X$ path-connected and finitely dominated, if $\wt K_0(\Z \pi_1 X) = 0$, then $X$ is homotopy equivalent to a finite complex.    In this paper, we are only interested in the case where $\pi_1X$ is torsion-free, and then  $\wt K_0(\Z \pi_1 X) $ is conjectured to vanish.  

We next turn to the notion of a Poincar\'e pair.  This could be a pair of spaces $(X,Y)$, but we wish to allow for the greater generality of a map $f : Y \to X$, in which case we will denote $H_*(M_f,Y)$ and $H^*(M_f,Y)$ by $H_*(X,Y)$ and $H^*(X,Y)$ and identify $H_*M_f$ and $H^*M_f$ with $H_*X$ and $H^*X$.  Here $M_f$ denotes the mapping cylinder of $f$.

\begin{definition}
 An {\em $n$-dimensional Poincar\'e pair} (a $PD_n$-pair) $(f: Y \to X,w)$ is a map of spaces $f$ and    an orientation character $w$ on $X$ so that
 
\begin{itemize}
 \item $X$ and $Y$ are finitely dominated and
 \item there exists  $[X] \in H_n(X,Y; \ol\Z)$ ``the fundamental class" so that for all $i$, the following maps are isomorphisms:
 \begin{align*}
  \cap [X] & :  \ol{H^i(X; \zpx)} \xrightarrow{\cong} H_{n-i}(X,Y; \ol\zpx) \\
  \cap [X] & :  \ol{H^i(X,Y; \zpx)} \xrightarrow{\cong} H_{n-i}(X; \ol\zpx).
 \end{align*}
 
 \item $(Y, w|_{\pi_1Y})$ is an $(n-1)$-dimensional Poincar\'e complex with fundamental class $\partial [X]$ where $\partial : H_{n}(X,Y;\ol \Z) \to H_{n-1}(Y;\ol \Z)$ is the boundary map in homology exact sequence of the pair.

\end{itemize}
\end{definition}

As above, there is an alternate definition of a Poincar\'e pair, replacing the two conditions by fundamental groups being finitely presented and requiring the duality isomorphisms with arbitrary twisted coefficients.

A compact topological manifold with boundary is a Poincar\'e pair.   

For any left $\Z\pi_1X$-module $B$, 
 \begin{align*}
  \cap [X] & :  {H^i(X; B)} \xrightarrow{\cong} H_{n-i}(X,Y; \ol B) \\
  \cap [X] & :  {H^i(X,Y; B)} \xrightarrow{\cong} H_{n-i}(X; \ol B),
 \end{align*}
 with corresponding statements for bimodules.

A group $\pi$ {\em has type FP} if the trivial $\Z\pi$-module $\Z$ has a finite length resolution 
by finitely generated projective (left) $\Z\pi$-modules.  A $K(\pi,1)$ is finitely dominated if and only if $\pi$ is finitely presented and has type FP.

A group $\pi$ is a $PD_n$-group (a {\em Poincar\'e duality group
of dimension $n$}) if $\pi$ has type FP and $H^i(\pi; \Z\pi)$ is infinite cyclic as an abelian group for $i = n$ and is zero otherwise.
Since $\Z\pi$ is a $(\Z\pi,\Z\pi)$-bimodule, 
$H^n(\pi; \Z\pi)$ is a right $\Z\pi$-module.   
The corresponding homomorphism $w: \pi \to \Z^\times$ is called
the {\em orientation character} of the Poincar\'e duality group $\pi$.
It is not a priori clear, but is nonetheless true, that $H_n(\pi; \ol \Z)$ is infinite cyclic and that the cap product with a generator gives an isomorphism
$$
\ol{H^i(\pi ; \zp)} \xrightarrow{\cong} H_{n-i}(\pi ; \ol\zp).
$$
Thus a finitely presented group $\pi$ is a $PD_n$-group if and only if  $K(\pi,1)$ is a Poincar\'e complex.   If $\rho$ is a subgroup of finite index in a torsion-free group $\pi$, then $\rho$ is a $PD_n$-group if and only if $\pi$ is.

Poincar\'e duality groups were introduced independently by Johnson and Wall \cite{JW} and Bieri \cite{Bieri}; 
a basic reference is the book of Brown \cite{Brown}.

  A closed aspherical $n$-manifold has fundamental group which is finitely presented and $PD_n$. The converse question  is still open:   is every finitely presented $PD_n$-group the fundamental group of a closed, aspherical $n$-manifold?

A closed Riemannian manifold with non-negative curvature is aspherical.   A torsion-free virtually polycyclic group is the fundamental group of a closed aspherical manifold \cite{AJ}.  Some remarks on aspherical 4-manifolds are given in \cite{DH2}; our focus there is on the case of manifolds with boundary.

\section{Poincar\'e duality groups with boundary}   \label{pdgb_section}

The purpose of this section is to introduce the analogue of a Poincar\'e duality group associated to a compact aspherical manifold with boundary.    The first thing to realize is that the boundary of such a manifold need not be aspherical  $(D^n,S^{n-1})$, or even connected $(T^{n-1} \times I, T^{n-1} \sqcup T^{n-1})$.   The second thing is that for a compact aspherical $n$ manifold $M$ with fundamental group $\pi$,
$$
H^i(M,\partial M; \Z\pi) \cong H_{n-i}(M; \Z\pi) \cong H_{n-i}(\text{pt}).
$$

\begin{definition} \label{pdgb}
A {\em Poincar\'e duality group with boundary of dimension $n$} is a triple $(\pi,Y,c)$ consisting of a group $\pi$ of type FP, a Poincar\'e complex $Y$ of dimension $n-1$, and a homomorphism $c: \pi_1Y \to \pi$, such
 that $H^i(\pi,Y;\zp)$ is infinite cyclic as an abelian group for $i = n$ and is zero otherwise.   
\end{definition}

Here $c$ can be interpreted as a homomorphism of  groupoids, or equivalently as a collection of group homomorphisms $c = \{c_\alpha : \pi_1(Y_\alpha,y_\alpha)\to \pi\}$ where $y_\alpha \in Y_\alpha$ is a choice of base point for each path-component $Y_\alpha$ of $Y$.   Here $H^i(\pi,Y;\Z\pi)$ is interpreted as $H^i(M(c),Y;\zp)$ where $M(c)$ is the mapping cylinder of a map $Y \to K(\pi,1)$ realizing $c$.

Since $H^n(\pi,Y;\Z \pi)$ is a $\Z\pi$-module, a Poincar\'e duality group with boundary determines an orientation character $w : \pi \to \{\pm 1\}$.

The theorem below is a slight generalization of Theorem \ref{PP}.   (It is a slight generalization since we only assume that $\pi$ is type FP, and not that $K(\pi,1)$ is finitely dominated.)
\begin{theorem} \label{PDn_theorem}
 Let $(\pi,Y,c)$ be a Poincar\'e duality group with boundary of dimension $n$. Then 
\begin{enumerate}
\item $H_n(\pi,Y;\ol \Z) \cong \Z$.   A choice of generator is called a {\em fundamental class} and is denoted by $[\pi,Y,c]$.
\item  Capping with a fundamental class gives isomorphisms
 \begin{align*}
  \cap [\pi,Y,c] & :  \ol{H^i(\pi; \Z\pi)} \xrightarrow{\cong} H_{n-i}(\pi,Y; \ol {\Z\pi}) \\
  \cap [\pi,Y,c] & :  \ol{H^i(\pi,Y; \Z\pi)} \xrightarrow{\cong} H_{n-i}(\pi;\ol {\Z\pi}).
 \end{align*}
\end{enumerate}
Thus if $\pi$ is finitely presented, $(M(c),Y)$ is a $n$-dimensional Poincar\'e pair.
 \end{theorem}
 
 For a ring $R$ and a left $R$-module $M$, then $M^* = \Hom_R(M,R)$ is a  right $R$-module.  In fact $M \mapsto M^*$ defines a contravariant functor from left $R$-modules to right $R$-modules.   If $C_{\bullet}$ is a chain complex of left $R$-modules, let $C^\bullet = \Hom_R(C_\bullet,R)$ be the dual cochain complex of right $R$-modules.
 
 For a group $\pi$ with orientation character $w$,  the group ring $\Z\pi$ has an involution, so we can do better.   If $M$ is a left module, then the {\em conjugate dual} $L^\dagger = \overline{L^*}= \overline{\Hom_{\mathbb{Z}\pi}(L,\mathbb{Z}\pi)}$ is a left module.  If $L$ is finitely generated projective, then so is $L^\dagger$, and furthermore the module map $L \to L^{\dagger \dagger}$ given by $l \mapsto (\phi \mapsto \ol{\phi(l)})$ is an isomorphism.  Thus the conjugate dual restricted to the category of finitely generated projective left modules is a contravariant endofunctor whose square is naturally isomorphic to the identity.

\begin{lemma} \label{evaluation lemma}
\begin{enumerate}

\item Let $C_\bullet$ be a left $\Z\pi$-chain complex.  Let $\ol{C^\bullet}$ be the conjugate dual cochain complex.   Let $B$ be a left $\Z\pi$-module.   Then for all $i$,  there is an evaluation map 
$$
\ev: H_i(\ol B \otimes_{\Z\pi}  C_\bullet)  \to \Hom_{\Z\pi}(H^i\ol{C^\bullet},B ).
$$

\item Let $C_\bullet$ be a left $\Z\pi$-chain complex of finitely generated projective modules so that $C_i = 0$ for $i > n$.   Then the evaluation map 
$$
\ev: H_n(\ol B \otimes_{\Z\pi}  C_\bullet)  \to \Hom_{\Z\pi}(H^n\ol{C^\bullet},B )
$$
 is an isomorphism of abelian groups.
\end{enumerate}
\end{lemma}

\begin{proof}
(1) Fixing $B$, there is a natural transformation of functors from left $\Z\pi$-modules to abelian groups $\psi_L : \ol B \otimes_{\Z\pi} L\to \Hom_{\Z\pi}(L^*,B )$ given by $b \otimes c  \mapsto (\alpha \mapsto \ol {\alpha(c)}b)$.  
%
Hence there is a chain map $\ol B \otimes_{\Z\pi} C_\bullet \to \Hom_{\Z\pi}(\ol{C^\bullet}, B)$.  And for the cochain complex $\ol{C^{\bullet}}$, there is a graded map $H_*(\Hom_{\Z\pi}(\ol{C^\bullet}, B)) \to \Hom_{\Z\pi}(H^*\ol{C^\bullet},B)$, which sends a homology class $[a: C^i \to B]$ to the map $[\alpha] \mapsto a(\alpha)$.

(2) If $L$ is finitely generated projective, then $\psi_L$ is an isomorphism.  Hence $\ol B \otimes C_\bullet \to \Hom_{\Z\pi}(\ol{C^\bullet}, B)$ is an isomorphism of chain complexes.  The map  $H_n( \Hom_{\Z\pi}(\ol{C^\bullet}, B))  \to \Hom_{\Z\pi}(H^n\ol{C^\bullet},B )$ is an isomorphism by the left exactness of Hom. Indeed if $C^{i} = (C_{i})^{\dagger}$, then the exactness of 
$$
0 \leftarrow H^{n}\ol{C^{\bullet}} \leftarrow C^{n} \leftarrow C^{n-1}
$$
implies the exactness of 
$$
0 \to \Hom_{\Z\pi}( H^{n}\ol{C^{\bullet}},B) \to \Hom_{\Z\pi}( C^{n},B) \to \Hom_{\Z\pi}( C^{n-1},B)
$$
\end{proof}

For $\alpha \in H^i\ol{C^\bullet}$ and $z \in H_i(\ol B \otimes_{\Z\pi}  C_\bullet)$, define the Kronecker pairing 
$$
\langle \alpha, z \rangle := (\ev z) \alpha  \in B.
$$

\begin{proof}[Proof of Theorem \ref{PDn_theorem}]

Let $M(c)$ be the mapping cylinder of a map $Y \to K(\pi,1)$ realizing $c : \pi_1Y \to \pi$.  Let $\wt Y \to Y$ be the  $\pi$-cover induced by the universal cover $\wt{M(c)} \to M(c)$.
Let $B$ be a $(\Z\pi,\Z\pi)$-bimodule.   Define the left $\Z\pi$-chain complexes,
\begin{align*}
C_\bullet(\pi;B) & = B \otimes_{\Z\pi} S_\bullet(\wt{M(c)})\\
C_\bullet(\pi,Y;B) & = B \otimes_{\Z\pi} S_\bullet(\wt{M(c)},\wt Y)
\end{align*}
and the right $\Z\pi$-cochain complexes,
\begin{align*}
C^\bullet(\pi;B) & = \Hom_{\Z\pi}(C_\bullet(\pi;B) ,\Z\pi)\\
C^\bullet(\pi,Y;B) & =\Hom_{\Z\pi}(C_\bullet(\pi,Y;B) ,\Z\pi).\ 
\end{align*}

A chain complex $P_\bullet$ of left $\Z\pi$-modules has {\em type FP} if each $P_i$ is projective and if $\oplus_i P_i$ is finitely generated.  A chain complex has support in $[a,b]$ if $P_i = 0$ for $i \not \in [a,b]$.     The proof of Corollary 5.1 of \cite{Wall_FII} shows that if $P_\bullet$ is a chain complex of type FP which vanishes in negative degrees, and if $H^i(\Hom_{\Z\pi}(P_\bullet,\Z\pi)) = 0$ for $i > N$, then $P_\bullet$ has the chain homotopy type of a chain complex of type FP with support in $[0,N]$.

Since $(\pi,Y,c)$ is a Poincar\'e duality group with boundary of dimension $n$, the above paragraph shows that the chain complex $C_\bullet(\pi,Y;\Z\pi)$ has the chain homotopy type of a type FP  chain complex $C_\bullet$ supported in $[0,n]$.

Applying  Lemma \ref{evaluation lemma} with $B = \Z$  to $C_\bullet$ gives
\[
 H_n(\ol\Z \otimes_{\Z\pi} C_\bullet) \xrightarrow{\cong} \Hom_{\Z\pi}(H^n\ol C^\bullet, \Z) = \Hom_{\Z\pi}(\Z, \Z)
= \mathbb{Z}.
\]
and hence
$H_n(\pi,Y;\ol\Z)\cong \Z$. 
Let $[\pi,Y,c] \in H_n(\pi,Y;\ol \Z)$ be a generator.   Following the conventions of \cite{Hi20},  the cap product induces maps of left $\Z\pi$-chain complexes
\begin{align*}
 \cap[\pi,Y,c] & : \ol{C^{n-\bullet}(\pi,Y; \Z\pi )} \to  C_\bullet(\pi; \ol{\Z\pi}) \\
  \cap[\pi,Y,c] & : \ol{C^{n-\bullet}(\pi; \Z\pi )} \to  C_\bullet(\pi,Y; \ol{\Z\pi}) \\
    \cap\partial [\pi,Y,c] & : \ol{C^{n-\bullet}(Y; \Z\pi )} \to  C_\bullet(Y; \ol{\Z\pi}).
\end{align*}

We  claim that
$$
\cap [\pi, Y,c] : \ol{H^{n}(\pi,Y; \Z\pi)} \to  H_0(\pi; \ol\Z \otimes_{\Z\pi} \Z\pi) =H_0(\pi; \ol{\Z\pi}) 
$$
is an isomorphism of infinite cyclic groups. 
Let $\varepsilon : H_0(\pi; \ol{\Z\pi}  ) \to \Z \otimes_{\Z\pi} \ol{\Z\pi} = \Z$ be the augmentation, which is an isomorphism.   For $\alpha  \in H^n(\pi,Y; \Z\pi)$, 
$$
\varepsilon (\alpha \cap [\pi, Y]) = \langle \alpha, [\pi, Y] \rangle,
$$
see, for example, Proposition 3.21 of \cite{DK}.  Lemma \ref{evaluation lemma} (with $B = \Z$) implies that there is a $\alpha $ so that 
$\langle \alpha, [\pi, Y] \rangle = 1$.   It follows that $\cap [\pi, Y]$ is an epimorphism, hence an isomorphism, since the domain and codomain are infinite cyclic.

Thus the chain level cap product 
$$
\cap [\pi, Y] : \ol{C^{n-\bullet}(\pi,Y; \Z\pi)} \to C_\bullet(\pi; \ol{\Z\pi})
$$
gives a quasi-isomorphism of $\Z\pi$-chain complexes which have the chain homotopy type of finite projective chain chain complexes, hence is a chain homotopy equivalence.    It follows that the cap product isomorphism induces an isomorphism 
$$
 \cap[\pi,Y,c] : H^{n-i}(\pi,Y; \ol{\Z\pi} ) \xrightarrow{\cong}  H_{i}(\pi; \Z\pi).
$$
 Our conclusion then follows from Corollary C of \cite{KQS} (giving the other Poincar\'e-Lefschetz duality) and the fact that $\pi$ is finitely presented (giving that the spaces are finitely dominated).

\end{proof}

A standard result (see \cite{JW}, Bieri \cite{Bieri}, or 
 \cite{Brown}) states that if $\pi'$ is a $PD_n$-group and is finite index in $\pi$, then $\pi$ is a $PD_n$-group.    The corresponding result holds for Poincar\'e duality groups with boundary.
 
\begin{proposition}
Let $(\pi,Y,c)$ be a triple consisting of a group $\pi$, a space $Y$, and a homomorphism $c : \pi \to \pi_1Y$.   Let $\pi'$ be a finite index subgroup of $\pi$, and let $Y' \to Y$ be the induced cover and $c' : \pi' \to \pi_1Y'$ the induced map.  If $(\pi',Y',c')$ is a Poincar\'e duality  group with boundary of dimension $n$, then so is $(\pi,Y,c)$.
\end{proposition}

The proofs of the $PD_n$-result in the above three references can be  modified to give a proof of the above proposition.

Alternatively, Theorem  \ref{PDn_theorem} shows that $(M(c'),Y')$ is a Poincar\'e pair, so its finite cover $(M(c),Y)$ is a Poincar\'e pair by Theorem H of \cite{KQS}, and hence $(\pi,Y,c)$  is a Poincar\'e duality group with boundary by Poincar\'e-Lefschetz duality.

\section{The Borel Conjecture for compact aspherical manifolds with boundary}  \label{uniqueness}

The {\em Borel Conjecture for a compact aspherical manifold $M$} states that any homotopy equivalence $h : W \to M$ from a compact manifold $W$ which restricts to a homeomorphism $h : \partial W \xrightarrow{\cong} \partial M$ on the boundary is homotopic to a homeomorphism with the homotopy fixed on the boundary.  

The Borel Conjecture has been proven for a wide class of manifolds $M$; it is possible, even likely, that it holds in general.   

The following theorem is quite well-known (see e.~g.~\cite{W23}); however we will outline it since the corresponding discussion of the  Existence Conjecture is not well-known and we will need the set-up.   The reader is advised to skip the proof, but to read the end of this section where we discuss applications and variants of the Borel Conjecture.

To apply the FJC to the Borel Conjecture, we need to define the geometric structure set.   Let $(X,Y)$ be a CW-pair where $X$ has a single component.   
An {\em manifold structure on $(X,Y)$} is a  homotopy equivalence $h : W \to X$ which restricts to a  homotopy equivalence on $\partial W \to Y$, where $W$ is a compact topological manifold with boundary.
An {\em manifold structure on a CW-pair $(X,Y)$ relative to $Y$} (or just a structure rel $Y$) is a manifold structure $h: W \to X$  which restricts to a homeomorphism $h : \partial W \to Y$.  Two  structures   $h: W \to X$ and $h': W' \to X$ relative to $Y$ are {\em equivalent}  if there is a homeomorphism $F :W \to W'$ so that $h \simeq h' \circ F \rel \partial W$, i.e.~ $h|_{\partial  W} = h' \circ F |_{\partial  W}$, and there is a homotopy $H : W \times I \to X$ so that $H_0 = h$, $H_1 = h' \circ F$ and $H_t|_{\partial W} = h|_{\partial W}$ for all $t \in I$.
In other words, the diagram below commutes up to homotopy, and where the homotopy is constant restricted to the boundary.
$$
\begin{tikzcd}
(W, \partial W) \arrow[dd,"F"] \arrow[dr,"h"] & \\
& (X,Y) \\
(W', \partial W') \arrow[ur,"h'" ']
\end{tikzcd}
$$
The set of equivalence classes of manifold structures on $(X,Y)$ is called the rel $Y$ structure set and is denoted
$
\cS(X \rel Y).
$
The Borel Conjecture for $M$ is precisely the statement that $\cS(M \rel \partial M)$ is a singleton (represented by the identity map.)   There is also a variant  structure set $\cS(X,Y)$ which allows $h$ and $h'$ to be homotopy equivalences of pairs.   

There are also decorated versions of the rel $Y$ structure set, most notably $\cS^s(X \rel \partial Y)$ where elements are represented by simple homotopy equivalences $h : W \to X$ which restrict to a homeomorphism $\partial W \to Y$.    If $\Wh(\pi_1X) = 0$, then $\cS^s(X \rel \partial Y) = \cS(X \rel \partial Y)$

We now give the proof of Theorem \ref{BC} from the introduction.

\begin{proof}[Proof of Theorem \ref{BC}]  We will start with the discussion of the high-dimensional case \ref{BC}(2).  
For any group $\g$ with orientation character $w : \g \to \{\pm 1 \}$, there are assembly maps
\begin{align*}
H_*(B\g; \bfK) & \to K_*(\Z \g), \\
H_*(B\g,w; \bfL) & \to L_*(\Z \g,w).
\end{align*}
Here the Quillen spectrum $\bfK$ has homotopy groups $K_n(\Z)$ and the Sullivan-Quinn-Ranicki spectrum $\bfL$ has homotopy groups $L_n(\Z)$ for $n \in \Z$.  In the nonorientable case ($w \not = 1)$, $\Z \g$ is given the $w$-twisted involution $\sum a_g g \mapsto \sum a_g w(g) g^{-1}$ and the generalized homology has twisted coefficients, see, for example, Appendix A of \cite{R92}.
The FJC in $K$-theory and $L$-theory for a torsion-free group $\g$ states that the assembly maps are isomorphisms.   (The FJC for groups with torsion is much more complicated, but we will not need this case.)   The Farrell-Jones Conjecture in $K$-theory for a torsion-free group $\g$ implies that $\Wh(\g)$, $\wt K_0(\Z\g)$, and $K_{-i}(\Z\g)$ for $i >0$ all vanish; this is the only consequence we need.

The FJC's in $K$- and $L$-holds for elementary amenable fundamental groups of compact aspherical 4-manifolds by Corollary 6 of \cite{DH2}, hence we do not include the validity of the FJC conjecture in part (1) of the statement of the theorem.  

A technicality is that $L$-groups come with $K$-theory ``decorations", for example there are $L$-groups $L_n^s(\Z \g,w) \to L^h_n(\Z \g,w) \to L^{\langle - \infty \rangle}_n(\Z \g,w)$.   The classification of manifolds up to simple homotopy uses the $s$-decoration, up to homotopy uses the $h$-decoration, and the FJC uses the $-\infty$ decorations.    But, assuming the FJC-conjecture in $K$-theory for a torsion-free group, this is irrelevant, because all decorations give the same $L$-group.   The easiest $L$-theory to define is $L^h_*(\Z \g,w)$, the $L$-theory based on free modules.   

We now review algebraic surgery exact sequences.     Let $\bfL\langle 1 \rangle \to \bfL$ be the 1-connective cover, which has the property that the map $\pi_n(\bfL\langle 1 \rangle) \to \pi_n(\bfL)$ is an isomorphism for $n \geq 1$ and $\pi_n(\bfL\langle 1 \rangle) = 0$ for $n \leq 0$.  By taking homotopy groups of the mapping cofibers of spectra level assembly maps,  for any connected  CW-complex $X$ with fundamental group $\g$ and orientation character $w$, there is  a commutative diagram of abelian groups with exact rows (see, for example, \cite[Definition 14.6]{R92})
$$
\begin{CD}
\cdots @>>> \cS^{\langle 1 \rangle}_{n+1,w}(X) @>>> H_n(X,w; \bfL\langle 1 \rangle) @>>> L_n(\Z\g,w)  @>>> \cdots \\
@. @VVV @VVV @| @.\\
\cdots @>>> \cS_{n+1,w}(X) @>>> H_n(X,w; \bfL) @>>> L_n(\Z\g,w)  @>>> \cdots 
\end{CD}
$$
Both the $L$-theory and the structure set are ``decorated", but this is irrelevant for a torsion-free group $\g$ provided the FJC in $K$-theory holds for $\g$.   In this case, 
the FJC conjecture in $L$-theory is precisely the statement that $\cS_{*,w}(B\g)=0$ in all degrees.

The main result of Ranicki's book \cite[Theorem 18.5]{R92} is that there is a bijection of sets $\cS^s(M \rel \partial M) \cong \cS^{\langle 1 \rangle,s}_{n+1,w}(M)$ for any compact manifold $M$ with boundary with $\dim M \geq 5$ or with $\dim M = 4$ and  ``good" fundamental group i.e. when $\pi_1M$ is finitely generated elementary amenable.   The decoration $s$ is dropped when the Whitehead group vanishes.

The commutative ladder above gives a long exact sequence 
\begin{multline} \label{one-connective-cover}
\cdots \to H_{n+1}(X,w; \bfL/\bfL\langle 1 \rangle) \to 
\cS^{\langle 1 \rangle}_{n+1,w}(X) \\ \to \cS_{n+1,w}(X) \to H_n(X,w; \bfL/\bfL\langle 1 \rangle) \to \cdots
\end{multline}

Note that if $n = \dim M$, the AHSS shows that $H_{n+1}(M,w; \bfL/\bfL\langle 1 \rangle) = 0$, so that the map from the 1-connective structure set to the structure set is injective.   

We are now at the end of the proof of \ref{BC}(2).  Let $M$ be a compact aspherical $n$-manifold as in the statement of the theorem.  The FJC's in $K$- and $L$-theory guarantee that $K$-theory decorations are irrelevant and that $\cS_{n+1,w}(M) = 0$.   Thus $\cS_{n+1,w}^{\langle 1 \rangle}(M) \cong \cS(M \rel \partial M)$ is also trivial.   

Next comes the four dimensional case.    We have classified the elementary amenable groups which are the fundamental groups of compact aspherical 4-manifolds with boundary (see Theorem A of \cite{DH2}) and all of these satisfy the FJC in $K$- and $L$-theory and all are good groups in the sense of \cite{FQ} and \cite{BKKPR}, hence the high-dimensional reasoning proof goes through and proves  \ref{BC}(1).  
\end{proof}

\begin{remark}
As stated in the introduction, Theorem \ref{BC}(2) is well-known to the experts.   The case of abelian fundamental group is contained in \cite[Essay V, Appendix C]{KS}.   Theorem \ref{BC}(2) was discussed in \cite{FRR} and the proof was given in \cite{L}, although stated for the closed case.

Some cases of Theorem \ref{BC}(1) are covered by \cite[Theorem on page 205]{FQ}.

Now we discuss the status of the Borel Conjecture.   It holds classically if $\dim M \leq 2$.   If $\dim M = 3$ and $M$ is orientable, the Borel Conjecture is a consequence of Perleman's Geometrization Theorem,  see \cite[Theorem 0.7]{KL} for references in the closed case.   Although we expect the Borel Conjecture to hold for nonorientable 3-manifolds, the proofs are unclear, due to both geometrization and due to three manifold decomposition theory.

The Borel Conjecture is not known in dimension 4 for any group that is not good.
 For example, it is not known if the Borel Conjecture holds for the product of a 2-torus with a surface of genus 2.   
 
 There are weaker versions of the Borel Conjecture in dimension four which hold when the FJC holds in $K$- and $L$-theory for that fundamental group - namely that the Borel Conjecture holds up to stable diffeomorphism.
 
In dimensions greater than four, Theorem \ref{BC}(2) shows that the validity of the Borel Conjecture depends on the FJC.  
A definitive list of groups for which the FJC in $K$- and $L$-theory is known is given in the upcoming book of L\"uck, {\em Isomorphism Conjectures in $K$- and $L$-theory}.   It is fair to say that it is currently known for many groups, but not in general.   For example, it is not known for all elementary amenable groups.

 \end{remark}

\subsection{Homeomorphic boundaries} \label{subsection_hb}

The theme of this subsection is the following question.

\begin{question}
If $W$ and $M$ are compact aspherical manifolds with isomorphic fundamental groups and homeomorphic boundaries, are $W$ and $M$ homeomorphic?
\end{question}

A classical example where the answer is no is given by the exteriors of the granny knot and the square knot, with torus boundary \cite{Fox}.   A four-dimensional example where the answer is no is given in Example \ref{nonhomeomorphic} (see also \cite{DH2}).

To reduce the question to the Borel Conjecture, we need to use elementary obstruction theory. 

For a collection of groups $H = \{H_\alpha\}$ and a group $G $, define an equivalence relation $\simeq$ on $\Hom(H,G):= \prod_\alpha \Hom(H_\alpha,G)$ by $\phi \simeq \phi'$ if there exist $\{g_\alpha \in G\}$ so that for all $\alpha$ and for all $h \in H_\alpha$, one has $\phi_\alpha (h) = g_\alpha \phi'_\alpha (h)  g_\alpha^{-1}$.   For a space $X$ with path-components $\{X_\alpha\}$, and for a path-connected space $Y$, there is a well-defined function
$$
\pi_1: [X,Y] \to \Hom(\pi_1X,\pi_1Y)/\simeq,
$$
independent of the choice of base-points, where $\pi_1X = \{\pi_1X_\alpha\}$.

Elementary obstruction theory (see \cite[Theorems 7.26 and 6.57]{DK}) proves the following lemma.

\begin{lemma} \label{ot}
Let $X$ be a CW-complex and $Y$ be aspherical.    Then
$$
\pi_1 : [X,Y] \to \Hom(\pi_1X,\pi_1Y)/\simeq
$$ 
is a bijection.
\end{lemma}

Assuming the Borel Conjecture, the question above is reduced to algebra in the following proposition.

\begin{proposition}  \label{BC_ot}
Let $W$ and $M$ be compact aspherical manifolds.   Suppose the Borel Conjecture holds for $M$.   Then $W$ is homeomorphic to $M$ if and only if there is a homeomorphism $h: \partial W \xrightarrow{\cong} \partial M$ and an isomorphism $\varphi: \pi_1 W \xrightarrow{\cong} \pi_1M$  
$$
\begin{tikzcd}
\pi_1 \partial W \arrow[r,"\pi_1(h)"] \arrow[d,"\pi_1(\partial W \hookrightarrow W)"] & \pi_1\partial M \arrow[d,"\pi_1(\partial M \hookrightarrow M)"] \\
\pi_1 W \arrow[r,"\varphi"] & \pi_1M
\end{tikzcd}
$$
so that  $\varphi \circ \pi_1(\partial W \hookrightarrow W) \simeq  \pi_1(\partial M\hookrightarrow M) \circ \pi_1(h) $. Furthermore, if these conditions hold then the homeomorphism $W \to M$ extends the given homeomorphism $\partial W \to \partial M$.
\end{proposition}

\begin{proof} Assume that   $\varphi \circ \pi_1(\partial W \hookrightarrow W) \simeq  \pi_1(\partial M\hookrightarrow M) \circ \pi_1(h)$.   Kirby and Siebenmann \cite{KS} showed that manifolds  have the homotopy type of a CW-complex. Using Lemma \ref{ot}, realize $\varphi$ by a homotopy equivalence $H : W \to M$ so that $\pi_1(H) = \varphi$.   By Lemma \ref{ot} again, 
$H \circ (\partial W \hookrightarrow W)\simeq (\partial M\hookrightarrow M) \circ h$.   Since $\partial W \hookrightarrow W$ is a cofibration, $H$ is homotopic to a map $K$ which restricts to $h$.    The Borel Conjecture then implies that $K$ is homotopic to a homeomorphism which restricts to $h$.
\end{proof}

A slogan summarizing the previous propostion is:
\medskip

{\em The homeomorphism type of a compact aspherical manifold is determined by its fundamental group, the homeomorphism type of its boundary, and the fundamental group system of the inclusion of its boundary. }

Here is a well-known corollary.

\begin{corollary}
  For any $n$,  the map from homeomorphism classes of compact contractible $n$-manifolds to the homeomorphism classes of homology $(n-1)$-spheres is injective.
\end{corollary}

Generalizations of the corollary are given in Section \ref{examples}.

\subsection{Observations about the Borel Conjecture}

We conclude this section with a few observations related to the Borel Conjecture.  

The Borel Conjecture for a manifold $M$ without boundary has a particularly appealing consequence -- any closed aspherical manifold with fundamental group $\pi_1M$ is not just homotopy equivalent to $M$, but actually homeomorphic to $M$.   In other words, any two closed $K(\pi,1)$-manifolds are homeomorphic.   

Now suppose that $M$ is a compact, aspherical manifold with boundary.    The first observation is that it is rare for the boundary to be aspherical (e.g. $(D^n,S^{n-1})$), although it is possible (e.g. $(T^n \times D^2, T^{n+1}))$.    The second observation is that the boundary is not determined by the group, for example Kervaire showed that any homology sphere of dimension greater than 3 bounds a compact contractible manifold, and Mazur constructed an example of a compact contractible 4-manifold with boundary a nontrivial homology 3-sphere.
In particular, it is not true that any two compact $K(\pi,1)$-manifolds are homeomorphic, even when the fundamental group is trivial.

We now consider variants of the Borel Conjecture.   

There is a version for open aspherical manifolds involving homotopy equivalences outside a compact set.   This can be studied via surgery theory and may be universally true.   However we will not pursue this.   

The Borel Conjecture does not hold in general in the smooth category.   Indeed if $\Sigma^7$ is an exotic 7-sphere then $T^7 \# \Sigma^7$ is homeomorphic, hence homotopy equivalence to $T^7$, but they are not diffeomorphic.   But dimension 4 is still a mystery.   Only recently \cite{DHHRS} constructed closed 4-smooth manifolds which are homotopy equivalent (in fact homeomorphic) which are not diffeomorphic.   But the case of the 4-torus is still open.

One could (falsely) conjecture that $\cS(M, \partial M) = *$ for $M$ compact aspherical, i.e. that any homotopy equivalence of compact manifold pairs $h : (W, \partial W) \to (M, \partial M)$ is homotopic to a homeomorphism.    But it is an exercise in the surgery exact sequence to see that for any $k \geq 2$, there is a compact manifold $W$ homotopy equivalent to $T^k \times D^3$ whose boundary is homotopy equivalent to $T^k \times S^2$,
but which is not itself homeomorphic to $T^k\times{D^3}$.
 (Hint: the relative $L$-theory is zero but the relative normal invariants are not.)    
 
 However, if $\cS(M \rel\partial M) = *$ and $\cS(\partial M) = *$, the homotopy extension property shows that $\cS(M, \partial M) = *$.  
  This is the case, for example, if both $M$ and $\partial M$ satisfy the hypothesis of Theorem \ref{BC}.   For example $\cS(T^n \times D^2, T^{n+1}) = *$.    But the boundary need not be aspherical, for example, spheres are topologically rigid, so  $\cS(D^n,S^{n-1})=*$.   A survey of topologically rigid closed manifolds is given by Kreck and L\"uck \cite{KL}.
 
 Theorem C of \cite{DH2} shows the following theorem.
 
\begin{theorem}
Let $M$ be a compact aspherical 4-manifold with elementary amenable fundamental group and orientable boundary.   Then  $\cS(M,\partial M) = *$.
\end{theorem}

\section{The Existence Conjecture for compact aspherical manifolds with boundary} \label{existence}

\medskip
The {\em  Existence Conjecture for a Poincar\'e pair $(X,Y)$ with $X$ aspherical and $Y$  a closed manifold} states that $\cS(X \rel Y)$ is nonempty, i.e., there is a commutative square 
$$
\begin{CD}
\partial M @>\cong>> Y\\
@VVV @VVV \\
M @>\sim>> X
\end{CD}
$$
where $M$ is a compact manifold, the upper horizontal map is a homeomorphism and the lower horizontal map is a homotopy equivalence.    This is equivalent to saying that $Y$ is the boundary of a compact aspherical manifold $M$ so that the inclusion of $Y$ in $X$ and the inclusion of $Y$ in $M$ induces the same map on the fundamental group.  Theorem \ref{PP} shows that Existence Conjecture is equivalent to Conjecture \ref{EC} when $Y$ is nonempty.
We will see that when $Y$ is nonempty, the  Existence Conjecture holds for a wide class of Poincar\'e pairs:  it is possible even likely, that it always holds in this case.    However, if $Y$ is empty the validity of the Existence Conjecture (Wall's Problem) is more questionable, although no counterexamples are known.

\begin{proof}[Proof of Theorem \ref{ET}]
The dimension 2 case is classical.
 The proof for $n \geq 4$ is based on the total surgery obstruction, expositions of which are given in \cite{Ra}, \cite{R92}, and \cite{KMM}.    If $(X,Y)$ is a finite Poincar\'e pair of dimension $n$ with closed manifold boundary, then one defines the {\em total surgery obstruction} $s(X \rel Y) \in \cS^{\langle 1 \rangle,h}_n(X)$. This is a homotopy invariant relative to $Y$ and vanishes if $\cS(X \rel Y)$ is nonempty.   If $n \geq 5$ or if $n=4$ and $\pi_1X$ is ``good", then $s(X \rel Y) = 0$ implies that  $\cS(X \rel Y)$ is nonempty.  Our sequel paper \cite{DH2} shows that if $(X,Y)$ is a four-dimensional Poincar\'e pair with $X$ aspherical and $\pi_1X$ elementary amenable, then $\pi_1X$ is good and the Farrell-Jones Conjectures in $K$ and $L$-theory hold for $\pi_1X$.
 
 We next recall that if the FJC in $K$-theory holds for a torsion-free group, then decorations are irrelevant for the algebraic structure groups and if the FJC in L-theory holds for the fundamental group of a finite dimensional aspherical complex $X$, then the nonconnective algebraic structure group $\cS_k(X)$ vanishes for all $k \in \Z$.

 Note that the exact sequence \eqref{one-connective-cover} and the Atiyah-Hirzebruch Spectral Sequence show that for any $n$-dimensional CW-complex $X$, there is an exact sequence
$$
0 \to \cS^{\langle 1 \rangle}_{n+1}(X) \to \cS_{n+1}(X) \to H_n(X;\Z^w) \to \cS^{\langle 1 \rangle}_{n}(X) \to \cS_{n}(X).
$$
For $(X,Y)$ a Poincar\'e pair with $X$ connected and $Y$ nonempty, $H_n(X;\Z^w) \\ \cong H^0(X,Y;\Z) = 0$.  
And the FJC hypotheses guarantee that $\cS_n(X) = 0$.   Thus the total surgery obstruction $s(X \rel Y) \in\cS^{\langle 1 \rangle}_{n}(X)= 0$.   Hence $\cS(X \rel Y)$ is nonempty as claimed.
\end{proof}

Theorem \ref{ET}(2) when $\pi$ is trivial or infinite cyclic is due to Freedman and Quinn \cite[Proposition 11.6A]{FQ}.

The  Existence Conjecture is open in dimension 3, but is known in dimensions less than 3 by work of Eckmann-M\"uller \cite{EM80}.

{\em Warning:}   It is possible that every aspherical Poincar\'e complex $X$ has the homotopy type of a closed manifold, but this is not implied by the standard conjectures of high dimensional topology.  Ranicki's total surgery obstruction gives an element $s(X) \in \cS^{\langle 1 \rangle}_n(X)$ which vanishes if and only if $X$ has the homotopy type of a closed manifold provided $n = \dim X > 4$.   Furthermore, if the FJC holds in $K$ and $L$-theory for $\pi_1X$, then $\cS_*(X) = 0$.  Thus $\cS^{\langle 1 \rangle}_n(X) = H_n(X;\Z^w)  \cong \Z$.    But computing this integer is very difficult \cite{Q}, although it has been computed in a nontrivial case \cite{BLW}.    Hence we don't elevate the question of whether an aspherical Poincar\'e complex has the homotopy type of a closed manifold to a conjecture, because we see no conceptual reason why it should be true. 

One could also ask whether a Poincar\'e pair $(X,Y)$ with $X$ aspherical must have the homotopy type of a compact manifold pair.   However, in this case, it is even more dubious, since there is an exact sequence
$$
H_n(X,Y;\Z^w ) \to \cS^{\langle 1 \rangle}_{n}(X,Y) \to \cS_{n}(X,Y).
$$
and $\cS_{n}(X,Y)$ might be nonzero.  In fact, the following example was 
suggested by Shmuel Weinberger.    It gives such a Poincar\'e pair with a nonreducible Spivak bundle.

\begin{lemma} \label{E_false_PP}
There is a Poincar\'e pair $(X,Y)$ with $X$ aspherical which is not homotopy equivalent to a compact manifold pair.   
\end{lemma}

\begin{proof}
We first construct a spherical fibration 
over a compact aspherical parallelizable manifold with boundary which has no stable topological reduction to a sphere bundle.   We start with a nonreducible spherical fibration
$$
S(\eta) \to B
$$
over a finite CW-complex.   For example, we could take $B= S^3$ and ``clutch" along the nontrivial element of $\pi_4S^2 = \pi_2(\text{hAut}_*(S^2))$ to construct a spherical fibration over $S^3$ whose fiber has the homotopy type of $S^2$.   (Here  $\text{hAut}_*(S^2)$ is the topological monoid of based self-homotopy equivalences of $S^2$.)  The exotic characteristic class of \cite{GS65} shows that this spherical fibration is not stably reducible.  

By work of Baumslag-Dyer-Heller \cite{BDH80}, there is a homology equivalence  $X' \to B$ with $X'$ a finite aspherical  simplicial complex.   Let $N = N(X' \hookrightarrow \R^K)$ be a regular neighborhood of a simplicial embedding of $X'$ in Euclidean space.   $N$ has the homotopy type of $X'$.    Thus there is a homology equivalence $h : N \to B$ where $N$ is a compact  aspherical parallelizable manifold with boundary.   Since $h$ is a homology equivalence, the pullback of the original spherical fibration along $h$ 
$$
\pi: S(\nu) \to N
$$
is still not stably reducible. 

Let $D(\nu)$ be the mapping cylinder of $\pi$.  Then we have a pair of fibrations over $N$,
$$
\begin{CD}
(D^k,S^{k-1}) @>>> (D(\nu),S(\nu))\\
@.  @VVV \\
@. N
\end{CD}
$$
Since $(D^k,S^{k-1})$ and $(N,\partial N)$ are both Poincar\'e pairs, Theorem G of \cite{KQS}
shows that
\begin{equation*}  
\begin{CD}
S(\nu)|_{\partial N} @>>> D(\nu)|_{\partial N}  \\
@VVV @VVV\\
S(\nu) @>>>  D(\nu)
\end{CD}
\end{equation*}
is a Poincar\'e triad, and hence
$$
(X,Y) := \left(D(\nu), S(\nu) \bigcup_{  S(\nu)|_{\partial N}}  D(\nu)|_{\partial N}\right)
$$
is a Poincar\'e pair with $X$ aspherical.  Furthermore the Spivak bundle of the Poincar\'e pair is classified by the composite $D(\nu) \to N \to BSG$ where the first map is a homotopy equivalence and the second map classifies the   nonreducible spherical fibration $\pi$.    In particular the Spivak bundle of the pair $(X,Y)$ is nonreducible, and hence $(X,Y)$ does not have the homotopy type of a compact manifold pair.  
\end{proof} 

\begin{remark}
The same procedure gives other examples.
By starting with a topological bundle over a finite CW-complex with no $PL$-reduction or a $PL$-bundle over a finite CW-complex with no $O$-reduction, one can produce a compact aspherical topological manifold (with boundary) with no $PL$-structure and a compact aspherical $PL$-manifold (with boundary) with no smooth structure.
\end{remark}

\begin{remark}
We have not discussed the following existence question: 
what are the possible fundamental groups of compact aspherical manifolds?
Davis, Januszkiewicz and Weinberger
\cite{DJW} show that any finite disjoint union 
of closed aspherical (n-1)-manifolds which is the boundary of a compact manifold, 
is the boundary of an aspherical compact manifold. 
This is particularly interesting when n=4, since all 3-manifolds are boundaries.
In \cite{dh3} we show that the corresponding question for $PD_4$-pairs with aspherical 
ambient space can largely be reduced to the ``algebraic" case of $PD_4$-pairs of groups
 (with all boundary components aspherical and $\pi_1$-injective), 
provided the necessary condition $H^2(\pi;\Z\pi)=0$ holds
\end{remark}

\section{The classification of compact, aspherical manifolds with abelian fundamental group} \label{examples}

In this section we first prove Theorem \ref{duality} which gives homological conditions for Poincar\'e pairs $(X,Y)$ with $X$ aspherical and $\pi_1X$ a duality group, and then use this to classify aspherical manifolds with abelian fundamental groups.  

In principal, Theorems \ref{ET}  and  \ref{duality} and \cite{DH2} list all homotopy types of compact aspherical 4-manifolds $(M,\partial M)$ with $\pi_1M$ elementary amenable, since by \cite{DH2} such groups are duality groups.

Examples of compact aspherical manifolds with duality fundamental group are    $T^k \times D^{n-k}$, and $(T^2 - \text{int } D^2) \times S^{n-2}$, and, for nonorientable examples, $T^k \times_{\Z/2} D^{n-k}$ where $\Z/2$ acts freely on $T^k$ and reverses orientation on $D^{n-k}$.  In some sense, our result imply that aspherical manifolds which homologically resemble these also exist.

Recall the hypotheses of Theorem \ref{duality}: $(X,Y)$ is a CW-pair with $X$ aspherical, $\pi = \pi_1X$ a duality group of dimension $k$, and $Y$ a nonempty Poincar\'e complex of dimension $n-1$.   Let $D = H^k(\pi;\Z\pi)$; this is a right $\Z\pi$-module.  Let $\ol X$ and $\ol Y$ be the induced $\pi$-covers of $X$ and $Y$.

\begin{proof}[Proof of Theorem \ref{duality}]
 Notice  that $k < n$.  

Suppose first that $(X,Y)$ is a Poincar\'e pair and that $k = n-1$.    Since $X$ is aspherical and $\pi = \pi_1X$ is a  $k$-duality group, the cohomology long exact sequence of the pair $(X,Y)$ shows that $H^i(Y;\Z\pi)$ is zero if $i \not =  n-1$, and gives a short exact sequence
\begin{equation*}
0 \to H^{n-1}(X;\Z\pi)   \to H^{n-1}(Y;\Z\pi)   \to H^{n}(X,Y;\Z\pi)   \to  0.
\end{equation*}
Let $w = w_1(X,Y)$.  Since $(X,Y)$ is a Poincar\'e pair, Poincar\'e duality shows that the above short exact sequence is isomorphic to the short exact sequence 
$$
0 \to \ol{H_{1}(X,Y;\Z\pi)} \to \ol {H_{0}(Y;\Z\pi)}  \to \ol{H_{0}(X;\Z\pi)}  \to 0
$$
and that $H_i(Y;\Z\pi) = H_i\ol Y$ vanishes for $i \not = 0$.

Now suppose $(X,Y)$ is a Poincar\'e pair and that $0 \leq k \leq n-2$. 
Since $X$ is aspherical, $\pi = \pi_1X$ is a  $k$-duality group, and $(X,Y)$ is a Poincar\'e pair, the cohomology long exact sequence of the pair $(X,Y)$ shows that $H^i(Y;\Z\pi)$ is zero if $i \not = k, n-1$, and
\begin{align*}
 H^{n-1}(Y;\Z\pi) & \xrightarrow{\cong} H^n(X,Y;\Z\pi) \\
  H^{k}(X;\Z\pi) & \xrightarrow{\cong} H^k(Y;\Z\pi)
\end{align*}
as right $\Z\pi$-modules.  Let $w = w_1(X,Y)$.  Since $Y$ is a Poincar\'e complex whose orientation character factors through $\pi$,
one sees that $H_i(Y;\Z\pi)= H_i\ol Y$ is zero for $i \not = 0, n-1-k$, and $H_{n-1-k}(Y;\Z\pi) = H_{n-1-k}\ol Y \cong \ol D$.  Note also that $H_{n-1-k}\ol X = 0$ since $X$ is aspherical.

We need some homological algebra before proving the converse.   
For a ring $R$, an $R$-module $A$, and an integer $j$, an {\em Eilenberg-MacLane complex of type $(A,j)$} is a bounded below  chain complex $\cC$ of projective $R$-modules so that $H_i\cC = 0$ for $i \not = j$ and $H_j\cC \cong A$.  In this case we will write $\cC = K(A,j)$.   If $\cC$ and $\cD$ are both Eilenberg-MacLane complexes of type $(A,j)$, then $\cC$ and $\cD$ are chain homotopy equivalent (inducing the identity on $A$).

Now suppose $R$ is a ring with involution, for example $\Z\pi$ with the $w$-twisted involution $\sum a_g g \mapsto \sum a_g w(g) g^{-1}$.  Then if $P$ is a left $R$-module, so is $P^\dagger := \ol{\Hom_R(P,R)}$ via $(r\varphi)(p) := \varphi(p)\ol r$.  If $P$ is finitely generated projective, so is $P^\dagger$.   Evaluation $p\mapsto (\varphi \mapsto \ol{\varphi(p)})$ is a left module map, which is an isomorphism if $P$ is finitely generated projective.  

Suppose, as in the hypothesis of the theorem, that $(X,Y)$ is a CW pair, $X$ is aspherical and finitely dominated with fundamental group a $k$-duality group $\pi$,  that $Y$ is a nonempty Poincar\'e complex of dimension $n-1$, and $w : \pi \to \{\pm1\}$ extends the orientation character of $w_1Y$.

Since $\pi$ is a $k$-duality group, $\ol{C^{-*}(\pi;\Z\pi)} = K(\ol D, -k)$, and its conjugate dual satisfies $C_*(\pi;\Z\pi) = K(\Z,0)$.  The hypothesis 1 and 2 give $C_*(X,Y; \Z\pi) = K(\ol D, n-k)$, and hence its conjugate dual $\ol{C^{-*}(X,Y; \Z\pi)} = K(\Z,-n)$.   Theorem C then implies that $(X,Y)$ is a Poincar\'e pair.
\end{proof}

\begin{remark}
If $k \leq  n-2$, then $H_{n-1-k} \ol Y \cong \ol D$ and $H_0\ol Y = \Z$.    If  $H_0\ol Y = \Z$, $Y$ is connected and $\pi_1 Y \to \pi_1 X$ is surjective.   If $k < n-2$, then $H_1\ol Y = 0$.   Hence  $\pi_1 Y \to \pi_1 X$ has perfect kernel and there is a unique extension of $w_1Y$ to $w$.
\end{remark}

\begin{corollary} \label{manifold duality group}
Suppose $(X,Y)$ is a pair of spaces with $X$ aspherical, $Y$ a nonempty closed $(n-1)$-manifold and $\pi = \pi_1X$ a $k$-duality group.  Suppose that $n = 4$ and $\pi$ is elementary amenable or that $n > 4$ and that the FJC in $K$ and $L$-theory hold for $\pi$.   Suppose also that $w: \pi \to \{\pm1\}$ extends $w_1Y$, and 
\begin{enumerate}
 \item $H_i\ol Y $ vanishes for $i \not = 0,n-1-k$; and 
 \item  $\ker (H_{n-1-k} \ol Y  \to H_{n-1-k} \ol X) \cong \ol D$ as left $\Z\pi$-modules, where we twist by $w: \pi \to \{\pm 1\}$.  
\end{enumerate}
Then $(X,Y)$ is realizable rel $Y$.

\end{corollary}

In the case of the trivial group (which is a duality group!) the result is well-known: the boundary of a compact contractible manifold is a homology sphere and the boundary map from the set of homeomorphism classes of compact contractible $n$-manifolds to homeomorphism classes of $(n-1)$-homology spheres is bijective (see the discussion in Section 21.3.2 of \cite{BKKPR}).   The deepest result used is the result of Freedman \cite{Fr} that every homology 3-sphere bounds a contractible 4-manifold.   

Freedman-Quinn (see 11.6A in \cite{FQ}) prove an analog of this in dimension 4 for  fundamental group $\Z$: they show that the boundary map from the set of homeomorphism classes of compact aspherical 4-manifolds with infinite cyclic fundamental group to the set of homeomorphism classes of closed 3-manifolds $N$ equipped with a map $\pi_1N \to \Z$ whose infinite cyclic cover has the homology of $S^2$ is a bijection.   As far as we know,  only the cases of trivial and infinite cyclic fundamental group have been considered in the literature.

Our first lemma determines the possibilities for the homology of the boundary of a compact aspherical $n$-manifold with abelian fundamental group.   One conclusion is that the boundary has the  $\Z[\Z^k]$-homology of  $S^{n-k-1} \times T^k$, equivalently that the boundary has a $\Z^k$-cover with the homology of $S^{n-k-1}$.  The proof of the lemma is omitted.

\begin{lemma} \label{boundary of BZ^k} 
Let $M$ be a compact aspherical $n$-manifold with abelian fundamental group $\pi$ and universal cover $\wt M$.   Then 
\begin{enumerate}
\item $\pi \cong \Z^k$ for $k \leq n$.
\item If $k = n$, then $M \cong T^n$.
\item  $H_*(\partial \wt M) \cong H_*S^{n-k-1}$.
\item If $k = n-1$ and $\partial M$ has two components and $M$, $\partial_1M$, and $\partial_2M$ are all $\Z[\Z^{n-1}]$-homology $T^{n-1}$'s,
\item If $k = n-1$ and $\partial M$ has one component,  then
$$
\pi_1 \partial M \to \pi_1 M \xrightarrow{w_1M} \{\pm 1\} \to 1
$$
is exact and   $M$ is a nonorientable manifold whose double cover is a $\Z[\Z^{n-1}]$-homology $T^{n-1}$.
\end{enumerate}
\end{lemma}

Next we determine the possible boundaries of compact aspherical 
$n$-manifolds with abelian fundamental group. 

\begin{corollary} (Existence) \label{existence free abelian} Let $N$ be a closed $(n-1)$-manifold equipped with a homotopy class of map $N \to T^k$ (equivalently, a homomorphism of groupoids $\pi_1N \to \Z^k$).   Suppose $k < n$.   Let $\ol N \to N$ be the corresponding $\Z^k$-cover.    Suppose 
$$
H_*\ol N \cong H_*S^{n-k-1}
$$
Then $N$ is the boundary of a compact aspherical manifold $M$ with fundamental group $\Z^k$, whose classifying map $M \to T^k$ restricts to the given one on the boundary.   
 \end{corollary}

\begin{proof}
For $n \geq 4$, this follows from Corollary \ref{manifold duality group}.

If $n =1$ and $k = 0$, then $N = S^0$ and $M = D^1$.  

If $n = 2$ and $k = 0$, then $N = S^1$ and $M = D^2$.  If $n =2$ and $k = 1$, then either $N = S^1 \amalg S^1$ or $N = S^1$ in which case $M$ is the annulus or the M\"obius strip, respectively. 

If $n = 3$ and $k = 0$, then $N = S^2$ and $M = D^3$. If $n =3$ and $k = 1$, then either $N$ is the 2-torus or the Klein bottle  in which case $M$ is $S^1 \times D^2$ or  $S^1 \times_{C_2} D^2$, respectively. 
If $n =3$ and $k = 2$, then either $N = T^2 \amalg T^2$ or $N = T^2$,
in which case $M$ is $T^2 \times I$ or $T^2 \times_{C_2} D^1= S^1 \times $ M\"obius strip, respectively. 
\end{proof}

\begin{remark} \label{rem:aspherical examples}
This remark discusses compact aspherical 4-manifolds with abelian fundamental group, focusing on connections to knot theory.  
By a $\Z[\Z^k]$-homology $S^{3-k}\times{T^k}$, we will mean a closed 
3-manifold $N$ with a $\Z^k$-cover $\ol N \to N$ so that 
$H_*\ol N \cong H_*S^{3-k}$.    
We have shown that for $k = 0,1,2,3$, such a $\Z[\Z^k]$-homology 
$S^{3-k}\times{T^k}$ bounds a compact aspherical 4-manifold with fundamental group $\Z^k$ (and conversely, the boundary of any such manifold is a $\Z[\Z^k]$-homology $S^{3-k}\times{T^k}$).

Freedman \cite{Fr} proved that a  $\Z$-homology 3-sphere bounds a compact contractible topological 4-manifold.   +1 surgery on a knot produces   
a $\Z$-homology 3-sphere.

Freedman and Quinn (see 11.6A in \cite{FQ}) showed that a \ $\Z[\Z]$-homology $S^2\times{S^1}$ bounds a compact 4-manifold having the homotopy type of a circle.  The easiest case is $S^2 \times S^1$ which bounds $D^3 \times S^1$.     Zero surgery on an Alexander polynomial 1 knot $K$ produces a  $\Z[\Z]$-homology $S^2\times{S^1}$, which must topologically bound a homology circle $W$.   One can deduce that $K$ is topologically slice by adding a 2-handle to $W$ along the meridian of $K$.

Our theorem shows that a  $\Z[\Z^2]$-homology $S^1\times{T^2}$
bounds a  compact 4-manifold having the homotopy type of a 2-torus.   
To the best of our knowledge, this is a new result.   
The easiest example of a compact aspherical 4-manifold with fundamental group $\Z^2$  is given by $D^2 \times T^2$, with boundary $T^3$.  More generally, rank two examples are given by  $D^2$-bundles over $T^2$ (classified by their Euler number) with boundary a nilmanifold.    These examples can be viewed as the 4-manifold presented by the Kirby diagram given by the Borromean rings with two dotted components (corresponding to two 1-handles added to $D^4$) and  integral Dehn surgery on the remaining component (corresponding to adding a 2-handle).    These examples can be generalized by  tying an Alexander 1 knot $K$ in a component of the Borromean rings and doing integral surgery on that component.  (By tying a knot in a component we mean replacing a $D^3$ which intersects the rings in an unknotted arc by another $D^3$ with a knotted arc  so that one of the components of the resulting link is the knot $K$.)    The 3-manifold boundary then satisfies the conditions of our theorem and hence bounds a compact 4-manifold which has the homotopy type of a 2-torus.

Here is another way of constructing rank 2 examples and simultaneously recovering the main result of \cite{D} that a 2-component link $L$ with Alexander polynomial 1 is topologically concordant to the Hopf link.  Let $M_L$ be the exterior of $L$ and $M_H$ be the exterior of the Hopf link.   Let $N = M_L \cup_{T^2 \amalg T^2} M_H$ where meridians and longitudes are identified.  Then $N$ is a $\Z[\Z^2]$-homology $S^1\times{T^2}$, which must, by our theorem, bound a compact aspherical 4-manifold $X$ with fundamental group $\Z^2$.   Adjoining a product of a neighborhood of the link with the interval, one arrives at the desired concordance
$$
X \bigcup_{(T^2~ \amalg ~ T^2) \times I} (S^1 \times D^2 \times I ~ \amalg ~ S^1 \times D^2 \times I)
\cong S^3 \times I$$

An orientable rank 3 example is given by $T^3 \times I$.   More generally, by a $\Z[\Z^3]$-homology $T^3$ we mean a closed, connected  3-manifold $N$ equipped with $\Z^3$-cover $\ol N \to N$ so that $\wt H_*(\ol N) =0$.   Then if $N_0$ and $N_1$ are a pair of $\Z[\Z^3]$-homology $T^3$'s, then $N_0 \sqcup N_1$ bounds a compact aspherical manifold with fundamental group $\Z^3$.

Examples are given by tying Alexander polynomial 1 knots in each of the components of the Borromean rings and doing zero Dehn surgery on each component to produce a $\Z[\Z^3]$-homology $T^3$ called $N_L$, and then $N_L \sqcup T^3$ then bounds a compact 4-manifold have the homotopy type of a 3-torus.

Note that if $L$ is a three component link with vanishing pairwise linking numbers, and $N_L$ is the three manifold obtain by doing 0-surgery on each component, then $N_L$ is a $\Z[\Z^3]$-homology $T^3$.
 Theorem 1.3 of Cha-Powell \cite{CP}  uses topological surgery to show that if $L$ is a three component link with vanishing pairwise linking numbers and trivial Arf invariants, then $N_L \coprod T^3$ bounds a compact aspherical manifold with fundamental group $\Z^3$.   Our result generalizes theirs by removing the hypothesis of trivial Arf invariant.    In fact if $L$ and $L'$ are both three component links with vanishing pairwise linking numbers, $N_L \coprod N_{L'}$ bounds a compact aspherical manifold with fundamental group $\Z^3$.    

If $N_L$ admits a orientation reversing free involution (which holds in the case of the Borromean rings), then our result shows that $N_L$ bounds a compact, nonorientable aspherical manifold with fundamental group $\Z^3$.   
\end{remark}

We have determined the possible boundaries of compact aspherical manifolds with abelian fundamental group.   We now ask whether the homeomorphism class of the boundary determines the homeomorphism class of the manifold.   In theory, this is an application of obstruction theory together with the Borel Conjecture.

\begin{theorem} (Uniqueness) \label{uniqueness free abelian} Suppose $k < n-2$. Two  compact aspherical $n$-manifolds with  fundamental group isomorphic to $\Z^k$   are homeomorphic if and only if their boundaries are homeomorphic.  In fact, any homeomorphism between the boundaries extends to a homeomorphism of the manifolds.
\end{theorem}

\begin{proof}
Let $W$ and $M$ be such manifolds and let $h : \partial W \to \partial M$ be a homeomorphism.   Since the Borel Conjecture holds for $M$ (see \cite{BFL}) and the fundamental groups of $W$ and $M$ are abelian, Proposition \ref{BC_ot} implies that the homeomorphism $h$ extends to a homeomorphism $M \to W$, if and only if there is an isomorphism $\varphi$ making the following diagram commute 
$$
\begin{tikzcd}
H_1(\partial W) \arrow[r,"h_*"]  \arrow[d] &  H_1(\partial M)  \arrow[d] \\
H_1(W) \arrow[r,dashed,"\varphi"] &  H_1 (M).
\end{tikzcd}
$$

To show that $\varphi$ exists, it suffices to show that the vertical maps in the above square are isomorphisms.     Since the arguments are parallel, we just consider the case of $H_1 (\partial M) \to H_1(M)$.   By Lemma \ref{boundary of BZ^k}, the relative homology  $H_*(M,\partial M ;\Z \pi)$ vanishes in degrees less than $n-k$.   By the Universal Coefficient Spectral Sequence (or by a direct argument), the  the relative homology  $H_*(M,\partial M ;\Z)$ vanishes in degrees less than $n-k$, which is greater than 2 by hypothesis.   Thus the vertical maps are isomorphisms. 
\end{proof}

Since the focus of our sequel paper \cite{DH2} is 4-manifolds, we state a corollary of Lemma \ref{boundary of BZ^k}, Corollary \ref{existence free abelian}, and Theorem \ref{uniqueness free abelian}.
Let $\cM^3$ be the set of homeomorphism classes of closed 3-manifolds.    For a fixed group $G$, let $\cM^4_G$ be the set of homeomorphism classes of compact aspherical 4-manifolds with fundamental group isomorphic to $G$.

The following corollary is a consequence of Lemma \ref{boundary of BZ^k} and Corollary \ref{existence free abelian}.

\begin{corollary} \label{boundaries_abelian}
For $G = 1, \Z, \Z^2, \Z^3, \Z^4$, let
$
\partial: \cM^4_G \to \cM^3
$
be the boundary map.
\begin{enumerate}
\item  $\partial(\cM^4_1) = \{ [N]  \mid H_* N \cong H_*S^3 \}$.
\item  $\partial(\cM^4_\Z) = \{ [N]  \mid \exists \Z\text{-cover } \ol N \to N \text{with }  H_*\ol N \cong H_*S^2\}$.
\item  $\partial(\cM^4_{\Z^2}) = \{ [N]  \mid \exists \Z^2\text{-cover } \ol N \to N \text{with }  H_*\ol N \cong H_*S^1\}$.
\item $\partial(\cM^4_{\Z^3}) = \{ [N]  \mid \exists \Z^3\text{-cover } \ol N \to N \text{with }  H_*\ol N \cong H_*S^0\}$.
\item $\partial(\cM^4_{\Z^4}) = \emptyset$.
\end{enumerate}
\end{corollary}

\begin{remark}
(1) is originally due to Freedman, \cite[Theorem 1.4']{Fr}.  As mentioned above, case (2) is due to Freedman-Quinn \cite[11.6A]{FQ}, where they also prove uniqueness.   Case (5) is a standard consequence of the Farrell-Jones Conjecture for $\Z^4$ and the fact that $\Z^4$ is a good group.   To the best of our knowledge, cases (3) and (4) are new.
\end{remark}

\begin{remark}
 Case (4) is a bit subtle.    There are two subcases: case (4a) where $N$ is disconnected and case (4b)  where $N$ is connected.   An example of case (4a) is the boundary of $T^3 \times I$.   An example of case (4b) is the boundary of $T^3 \times_{C_2} [-1,1]$ where $C_2$ acts freely on $T^3$ with quotient homeomorphic to $T^3$ and $C_2$ acts on $[-1,1]$ by multiplication by $-1$.   
 
Here is more on case (4a).   Assume $N$ is disconnected.   Then the existence of a $\Z^3$-cover $\ol N \to N$ with $H_*(\ol N) \cong H_*S^0$ is equivalent to $N = N_0 \sqcup N_1$ and where there exist $\Z^3$-covers $\ol N_i \to N_i$ with $\wt H_*(\ol N_i) = 0$  for $i = 0,1$.

Here is more on case (4b).   Assume $N$ is connected.   Then the existence of a $\Z^3$-cover $\ol N \to N$ with $H_*(\ol N) \cong H_*S^0$ is equivalent to the existence of a $\Z^3$-cover $\ol P \to N$ with $\wt H_*(\ol P) = 0$.

Furthermore, if $W$ is a compact aspherical 4-manifold with fundamental group $\Z^3$, then $W$ is orientable in case 4(a) and nonorientable in case (4b).   The orientation double cover in case (4b) is a manifold of type (4a).
\end{remark}

An obvious question is whether the boundary map in Corollary \ref{boundaries_abelian} is injective, in other words, does the homeomorphism type of boundary determine the homeomorphism type of the coboundary?
We saw in Theorem \ref{uniqueness free abelian}  that the answer is yes if $G = 1, \Z$.

The theorem below gives, given a closed 3-manifold $N$, an analysis of the homeomorphism classes of compact aspherical 4-manifolds with abelian fundamental group whose boundary is $N$.

\begin{theorem} \label{summary}
Let $W$ and $M$ be compact, aspherical n-manifolds with abelian fundamental group.   Let $h : \partial W \to \partial M$ be a homeomorphism.
\begin{enumerate}
\item $\pi_1W \cong \pi_1M \cong \Z^k$ for some $k \leq n$.
 \item  If there is an isomorphism $\varphi$ making the following diagram commute
\[
\begin{tikzcd}
H_1\partial W \arrow[r,"h_*"]  \arrow[d] &  H_1\partial M  \arrow[d] \\
H_1W \arrow[r,dashed,"\varphi"] &  H_1 M,
\end{tikzcd}
\]
then $h$ extends to a homeomorphism $H : W \to M$ with $H_*(H) = \varphi$.
 \item If $k < n-2$, then $H_1\partial W \to H_1W$  and $H_1\partial M \to H_1M$ are isomorphisms. 
Hence $h$ extends to a homeomorphism $H : W \to M$.
\item If $k = n-2$, then $h$ extends to a homeomorphism $H : W \to M$ if and only if 
\[
h_*(\ker (H_1\partial W \to H_1W) )= \ker (H_1\partial M \to H_1M).
\]
\item If $k = n-1$ and $\partial M$ is disconnected, then $\partial M$ and $\partial W$ have two components.    Let $h = h_0 \sqcup h_1 : \partial_0 W \sqcup \partial_1 W \xrightarrow{\cong} \partial_0 M \sqcup \partial_1 M$.   Then $H_1\partial_i M  \xrightarrow{\cong}  H_1M$ and $H_1\partial_i W  \xrightarrow{\cong}  H_1W$ for $i = 0,1$.
Furthermore, $h$ extends to a homeomorphism $H : W \to M$ if and only if if the composite 
\begin{multline*}
 H_1 M \xleftarrow{\cong}  H_1\partial_1 M \xleftarrow{\cong}  H_1\partial_1 W \xrightarrow{\cong} H_1 W \\ \xleftarrow{\cong}  H_1\partial_0 W \xrightarrow{\cong}  H_1\partial_0 M  \xrightarrow{\cong} H_1 M 
\end{multline*}
is the identity.
\item If $k = n-1$ and $\partial M$ is connected, then $H_1\partial M \to  H_1M$ and $H_1\partial W \to H_1W$ are both injective with image a subgroup of index 2.   Let $e(H_1\partial M \to  H_1M) \in H_1\partial M /2 H_1\partial M $ and  $e(H_1\partial M) \to  H_1M) \in H_1\partial W /2 H_1\partial W $ be the corresponding extension classes.  Then $h$ extends to a homeomorphism  $H : W \to M$ if and only if $h_*(e(H_1\partial M \to  H_1M)) = e(H_1\partial W \to  H_1W)$.
\item If $k = n$, then any isomorphism $H_1W \xrightarrow{\cong} H_1M$ is realized by a homeomorphism $H : W \to M$.
\end{enumerate}

\begin{remark}\label{surjective remark}
Given abelian groups $A$, $A'$, $B$, and $B'$, an isomorphism $\alpha$, and epimorphisms $\beta$ and $\gamma$ as pictured below, there is an isomorphism $\delta$ making the square commute if and only if $\alpha(\ker \beta) = \ker(\gamma)$.
$$
\begin{tikzcd}
A \arrow[r,"\alpha"]  \arrow[d,"\beta"] &  B  \arrow[d,"\gamma"] \\
A' \arrow[r,dashed,"\delta"] &  B'.
\end{tikzcd}
$$
\end{remark}

\begin{proof}
\noindent (1)   This follows from Lemma \ref{boundary of BZ^k}, parts (1) and (3).\\
\noindent (2)   The Borel Conjecture for $M$ holds since the fundamental group is abelian. (This is classical for $n \leq 2$, follows from Perelman and geometrization for $n=3$, and for $n \geq 4$, see \cite{BFL}).
The conclusion then follows from Proposition  \ref{BC_ot} and the fact that the fundamental group of $M$ is abelian.\\
\noindent (3)      Since the arguments are parallel, we will just show that $H_1 \partial M \to H_1M$ is an isomorphism.   By Lemma \ref{boundary of BZ^k}, the relative homology  \\$H_*(M,\partial M ;\Z \pi_1M)$ vanishes in degrees less than $n-k$.   By the Universal Coefficient Spectral Sequence (or by a direct argument),   the relative homology  $H_*(M,\partial M ;\Z)$ vanishes in degrees less than $n-k$, and hence in degrees 1 and 2.   Thus the  map is an isomorphism. 

\noindent (4)   When $k = n-2$, $H_*(\partial \wt M)\cong H_*S^1$ by Lemma \ref{boundary of BZ^k}.    Thus  $\pi_1\partial M \to \pi_1M$ is surjective since $\partial \wt M$ is path-connected.   Hence  $H_1\partial M \to H_1M$ is surjective with infinite cyclic kernel.  Likewise $H_1\partial W \to H_1W$ is surjective.   The conclusion follows from Remark \ref{surjective remark} and part (2).

\noindent (5)  If $k = n-1$, Lemma \ref{boundary of BZ^k} implies that $H_*\partial \wt M \cong H_*S^0$.  If $\partial M$ is disconnected, then  $\partial M$ and $\partial W$ have two components, so that $h = h_0 \sqcup h_1 : \partial_0 W \sqcup \partial_1 W \xrightarrow{\cong} \partial_0 M \sqcup \partial_1 M$.   Furthermore $H_1\partial_i \wt M \to H_1\wt M$ and $H_1\partial_i \wt W \to H_1\wt W$ are isomorphisms.   Hence $H_1\partial_i M \to H_1M$ and $H_1\partial_i  W \to H_1W$ are isomorphisms. 

 The conclusion follows from part (2).

\noindent (6)
This follows from the correspondence between $\Ext^1$ and extensions, but can be done explicitly without reference to $\Ext^1$.   We will be brief.    $\Ext^1(C_2,A) \cong A/2A$ for an abelian group $A$.  Given a short exact sequence 
$$
0 \to A \xrightarrow{\beta} A' \xrightarrow{\pi} C_2 \to 0,
$$
 the extension class $e(A \to A')  \in A/2A$ is given by $a + 2A$,  where there exists an $a' \in A'$ so that $i(a) = 2a'$ and $\pi(a')$ is nontrivial, and the extension class is independent of the choice of $a'$.     Then given a  square of abelian groups, 
$$
\begin{tikzcd}
A \arrow[r,"\alpha"]  \arrow[d,"\beta"] &  B  \arrow[d,"\gamma"] \\
A' \arrow[r,dashed,"\delta"] &  B'.
\end{tikzcd}
$$
with $\alpha$ an isomorphism and $\beta$ and $\gamma$ injections with index 2 images, then there is an isomorphism $\delta$ making the diagram commute if and only if $\alpha(e(A \to A')) = e(B \to B')$.   The conclusion then follows from part (2).

\noindent (7)    When $k = n$, the boundaries are empty and the fundamental groups are isomorphic, so we can take any isomorphism $\varphi$ in part (2).
\end{proof}

\begin{example} \label{nonhomeomorphic}
In cases (4), (5), and (6) the homeomorphism type of the cobounding compact aspherical manifold is not determined by  the homeomorphism type of  the boundary alone.   We shall sketch an example in case (4).  More details and similar examples for the remaining cases, using the same homology 3-torus, are given in \cite{DH2}.

Let $K_1, K_2$, and  $K_3$ be  distinct hyperbolic Alexander one knots, with exteriors $X_1$, $X_2$, and $X_3$.   The closed 3-manifold $N_L$ obtained by attaching the knot exteriors to the exterior  of the Borromean rings Bo (via gluing maps that identify meridians and longitudes of the knots with the meridians and longitudes of the components of $Bo$) is a homology 3-torus: $H_1N_L \cong \Z^3$ with a basis given by the images of the meridians; in fact it is a $\Z[\Z^3]$-homology 3-torus: $H_*(N_L;\Z[\Z^3]) \cong H_*(T^3;\Z[\Z^3])$.  Corollary \ref{existence free abelian} shows that there are compact aspherical manifolds $M_1, M_2$, and $M_3$ with fundamental group free abelian of rank 2, each with boundary $N_L$ so that $H_1N_L \to H_1M_i$ is an epimorphism whose kernel is generated by the $i$-th meridian.   

 The homeomorphism group  $\text{Homeo}(N_L)$ must preserve the JSJ decomposition $N_L = X(Bo) \cup X_1 \cup X_2 \cup X_3$, and so must preserve the meridianal basis of $H_1N_L$ (up to sign), since the hyperbolic pieces of the decomposition are pairwise nonhomeomorphic.   In particular, the three epimorphisms of $H_1N_L$ onto $\Z^2$ with kernel generated by a meridian are pairwise inequivalent under the action of $\text{Homeo}(N_L)$.  Then Theorem \ref{summary}(4) shows that the manifolds $M_1$, $M_2$, and $M_3$ are pairwise nonhomeomorphic.

Of course  in simple cases  the homeomorphism type of the boundary determines the homomorphism type of a cobounding compact aspherical 4-manifold with abelian fundamental group.   For example, we know that $T^3$ bounds a unique compact aspherical 4-manifold with fundamental group $\Z^2$ and as well as a unique compact aspherical 4-manifold with fundamental group $\Z^3$.   And $T^3 \coprod T^3$ bounds a unique compact aspherical 4-manifold with fundamental group $\Z^3$.

\end{example}

\end{theorem}

\bibliography{dh}{}
\bibliographystyle{alpha}

\end{document}